\documentclass[11pt]{article}

\usepackage[a4paper, margin=1in]{geometry}
\usepackage{amsfonts,amsmath,amssymb,amsthm,graphicx,fixmath,latexsym,color}
\usepackage{xcolor}
\usepackage{hyperref,epsfig}
\usepackage{algorithmic}
\usepackage{algorithm}
\usepackage{xspace}

\def\eps{{\varepsilon}}
\providecommand{\remove}[1]{}

\newcommand{\HD}{\ensuremath{\mathsf{HD}}\xspace}

\renewcommand{\Re}{\mathbb{R}}

\newcommand{\D}{\mathcal{D}}
\newcommand{\C}{\mathcal{C}}

\newcommand{\F}{\mathcal{F}}
\newcommand{\G}{\mathcal{G}}
\newcommand{\h}{\mathcal{H}}

\newtheorem{theorem}{Theorem}[section]
\newtheorem{lemma}[theorem]{Lemma}
\newtheorem{proposition}[theorem]{Proposition}

\newtheorem{definition}[theorem]{Definition}
\newtheorem{remark}[theorem]{Remark}

\newtheorem*{theorem*}{Theorem}
\newtheorem*{lemma*}{Lemma}
\newtheorem*{proposition*}{Proposition}

\newtheorem*{example*}{Example}

\begin{document}
\title{On Piercing Numbers of Families Satisfying the $(p,q)_r$ Property}

\author{Chaya Keller\thanks{Department of Mathematics, Ben-Gurion University of the NEGEV, Be'er-Sheva Israel. \texttt{kellerc@math.bgu.ac.il}. Research partially supported by Grant 635/16 from the Israel Science Foundation, the Shulamit Aloni Post-Doctoral Fellowship of the Israeli Ministry of Science and Technology, and by the Kreitman Foundation Post-Doctoral Fellowship.}
\and
Shakhar Smorodinsky\thanks{Department of Mathematics, Ben-Gurion University of the NEGEV, Be'er-Sheva Israel.
\texttt{shakhar@math.bgu.ac.il}. Research partially supported by Grant 635/16 from the Israel Science Foundation.}
}

\date{}
\maketitle

\begin{abstract}
The Hadwiger-Debrunner number $\HD_d(p,q)$ is the minimal size of a piercing set that can always be guaranteed for a family of compact convex sets in $\Re^d$ that satisfies the $(p,q)$ property. Hadwiger and Debrunner showed that $\HD_d(p,q) \geq p-q+1$ for all $q$, and equality is attained for $q > \frac{d-1}{d}p +1$. Almost tight upper bounds for $\HD_d(p,q)$ for a `sufficiently large' $q$ were obtained recently using an enhancement of the celebrated Alon-Kleitman theorem, but no sharp upper bounds for a general $q$ are known.

In~\cite{MS11}, Montejano and Sober\'{o}n defined a refinement of the $(p,q)$ property: $\F$ satisfies the $(p,q)_r$
property if among any $p$ elements of $\F$, at least $r$ of the $q$-tuples intersect. They showed that $\HD_d(p,q)_r \leq p-q+1$ holds for all $r>{{p}\choose{q}}-{{p+1-d}\choose{q+1-d}}$; however, this is far from being tight.

In this paper we present improved asymptotic upper bounds on $\HD_d(p,q)_r$ which hold when only a tiny portion of the $q$-tuples intersect. In particular, we show that for $p,q$ sufficiently large, $\HD_d(p,q)_r \leq p-q+1$ holds with $r = \frac{1}{p^{\frac{q}{2d}}}{{p}\choose{q}}$. Our bound misses the known lower bound for the same piercing number by a factor of less than $pq^d$.

Our results use Kalai's Upper Bound Theorem for convex sets, along with the Hadwiger-Debrunner theorem and the recent improved upper bound on $\HD_d(p,q)$ mentioned above.
\end{abstract}

\section{Introduction}
\label{sec:introduction}

Throughout this paper, $\F$ denotes a finite family of compact convex sets in $\Re^d$, $p,q \in \mathbb{N}$ satisfy $p \geq q \geq d+1$, and $|\F| \geq p$. $\F$ is said to satisfy the $(p,q)$ property if among any $p$ elements of $\F$ there is a $q$-tuple with a non-empty intersection. We say that $\F$ is \emph{pierced} by $S \subset \Re^d$ if any $A \in \F$ satisfies $A \cap S \neq \emptyset$. The smallest cardinality of a set that pierces $\F$ is called the \emph{piercing number} of $\F$. We call $\F$ \emph{$t$-degenerate} if all elements of $\F$ except at most $t$ can be pierced by a single point. Otherwise, $\F$ is called \emph{non-$t$-degenerate}.

\medskip The classical Helly's theorem asserts that if $\F$ satisfies the $(d+1,d+1)$ property (namely, if any $d+1$ elements of $\F$ have a non-empty intersection), then the piercing number of $\F$ is $1$.

In 1957, Hadwiger and Debrunner~\cite{HD57} considered a natural generalization of Helly's theorem to $(p,q)$ properties. Let $\HD_d(p,q)$ be the maximum piercing number taken over all families $\F$ of at least $p$ compact convex sets in $\Re^d$ that satisfy the $(p,q)$ property. Is $\HD_d(p,q)$ necessarily bounded for all $p \geq q \geq d+1$? (It is easy to see that $\HD_d(p,q)$ can be unbounded for $q \leq d$.)

Hadwiger and Debrunner showed that for all $p \geq q \geq d+1$ we have $\HD_d(p,q) \geq p-q+1$, and that equality is attained for any $(p,q)$ such that $q > \frac{d-1}{d}p +1$ (and in particular, in $\Re^1$ equality is attained for all $p \geq q \geq 2$). In a celebrated result from 1992, Alon and Kleitman~\cite{AK92} proved the Hadwiger-Debrunner conjecture, obtaining the upper bound $\HD_d(p,q) = \tilde{O}(p^{d^2+d})$. However, as mentioned in~\cite{AK92}, this bound is very far from being tight. The best currently known lower bound (implicitly implied by a result of Bukh et al.~\cite{BMN11}), is $\HD_d(p,q) = \Omega \left( \frac{p}{q} \log^{d-1} \frac{p}{q} \right)$.

Since the Alon-Kleitman theorem, several papers aimed at obtaining improved bounds on $\HD_d(p,q)$ for various values of $(p,q,d)$. The most notable result of this kind is by Kleitman et al.~\cite{KGT01} who showed that $\HD_2(4,3) \leq 13$ (compared to the upper bound of 345 obtained in~\cite{AK92}). Recently, it was shown in~\cite{KST16} that $\HD_d(p,q) \leq p-q+2$ for all $\eps>0$, $p \geq p_0(\eps)$, and $q>p^{\frac{d-1}{d}+\eps}$. The best currently known upper bound that holds for all $q$, $\HD_d(p,q) = \tilde{O}(p^{d \cdot \frac{q-1}{q-d}})$ (also shown in~\cite{KST16}), is apparently far from being tight.

In an attempt to obtain improved bounds on $\HD_d(p,q)$ by refining the $(p,q)$ property, Montejano and Sober\'{o}n~\cite{MS11} introduced the following definition: A family $\F$ is said to satisfy the $(p,q)_r$ property if among any $p$ elements of $\F$, at least $r$ of the $q$-tuples intersect. $\HD_d(p,q)_r$ is defined as the maximal piercing number taken over all families that satisfy the $(p,q)_r$ property. The main result of~\cite{MS11} is:
\begin{theorem}[\cite{MS11}]
\label{Thm:MS}
For any $d$,
\begin{equation}\label{Eq:Basic}
\HD_d(p,q)_r \leq p-q+1
\end{equation}
holds for all $r>{{p}\choose{q}}-{{p+1-d}\choose{q+1-d}}$.
\end{theorem}
The proof of Theorem~\ref{Thm:MS} uses a nice geometric argument.
As mentioned in~\cite{MS11}, the upper bound of Theorem~\ref{Thm:MS} is far from being tight. Moreover, the value of $r$ in the theorem is rather large -- almost all the ${{p}\choose{q}}$ $q$-tuples are required to intersect.

\medskip

In this paper we present improved upper bounds on $\HD_d(p,q)_r$. For $p,q$ sufficiently large (as function of $d$), our bounds hold already when $r$ is a tiny fraction of ${{p}\choose{q}}$. Our main result is the following:
\begin{theorem}\label{Thm:Main-UBT}
$\HD_d(p,q)_r$ satisfy:
\begin{enumerate}
\item For all $p \geq q \geq d+1$ and $r \geq \Theta_d \left({{\frac{d-1}{d}p}\choose{q-d}} {{\frac{p}{d}}\choose{d}} \right)$,
\[
\HD_d(p,q)_r \leq \min(p-q+1,\frac{p}{d}-1).
\]
\item For any $\eps>0$, any $p \geq q \geq d+1$ such that $p>p_0(\eps)$ and all $r \geq \Theta_{d,\eps} \left( \frac{p^{\left(\frac{d-1}{d}+\eps\right)q+1}}{(q-d)!} \right)$,
\[
\HD_d(p,q)_r \leq \min(p-q+1,p-p^{\frac{d-1}{d}+\eps}+2).
\]
\noindent Here, $\Theta_d(\cdot)$ hides a multiplicative factor that may depend on $d$.
\end{enumerate}
\end{theorem}
The latter bound on $r$ is not far from being tight, as an explicit example presented in~\cite{MS11} (which we recall below) yields a lower bound of $r=\Omega \left(\frac{p^{\left(\frac{d-1}{d}+\eps\right)(q-1)+1}}{(q-1)!}\right)$ for the same piercing number. The upper and lower bounds differ by a multiplicative factor of
$\frac{p^{\frac{d-1}{d}+\eps}(q-1)!}{(q-d)!}$, which is smaller than $pq^{d-1}$ for all $\eps \leq \frac{1}{d}$.

We note that for $p,q$ sufficiently large (as function of $d$), the condition in (1) is equivalent to $r \geq \frac{{{p}\choose{q}}}{c^q}$ for $c>1$ that depends only on $d$, and the condition in (2) (for $\epsilon=\frac{1}{2}$) is stronger than the condition $r \geq \frac{{{p}\choose{q}}}{p^{\frac{q}{2d}}}$ stated in the abstract. This means that the assertion of Theorem~\ref{Thm:Main-UBT} holds already when $r$ is an exponentially (in $q$) small fraction of ${{p}\choose{q}}$.

The proof of Theorem~\ref{Thm:Main-UBT} uses Kalai's Upper Bound Theorem for convex sets~\cite{K84}, combined with the Hadwiger-Debrunner theorem and the recent improved upper bound on $\HD_d(p,q)$ obtained in~\cite{KST16}.

\medskip

In view of Theorem~\ref{Thm:Main-UBT}(1), it is natural to ask whether a smaller value of $r$ is sufficient if we allow $\HD_d(p,q)_r$ to be larger than $\frac{p}{d}$ (but still smaller than $p-q$). We partially answer this question in the following generalization of Theorem~\ref{Thm:MS}.
\begin{theorem}\label{Thm:Main-Soberon}
For any $p \geq q \geq d+1$ and $0 \leq k \leq p-q-1$, denote by $m_0(k)$ the smallest integer $m$ such that ${{m+1}\choose{2}} \geq \frac{(p-q-k-1)(p-q+k+2)}{2}+1$. Let $\F$ be a non-$(p-q)$-degenerate family of compact convex sets in $\Re^d$ that satisfies the $(p,q)_r$ property, with
\begin{equation}
\label{Eq:Ex-Soberon1}
r \geq {{p}\choose{q}} - {{p-d+1}\choose{q-d+1}} + 1 + {{q-d-2+m_0(k)}\choose{q-d}} + {{q-d-1+m_0(k)}\choose{q-d+1}}.
\end{equation}
Then $\F$ can be pierced by at most $k+2$ points.
\end{theorem}
Note that in the case $k=p-q-1$, Theorem~\ref{Thm:Main-Soberon} reduces to Theorem~\ref{Thm:MS}. The proof of Theorem~\ref{Thm:Main-Soberon} uses a bootstrapping technique based on the technique pioneered by Montejano and Sober\'{o}n in~\cite{MS11}. The added value of Theorem~\ref{Thm:Main-Soberon} over Theorem~\ref{Thm:Main-UBT} is demonstrated well for small values of $(p,q)$. For example, for $(p,q,d)=(6,3,2)$, Theorem~\ref{Thm:Main-UBT} (actually, its proof) implies that $r=17$ is sufficient for assuring piercing by 2 points. Theorem~\ref{Thm:Main-Soberon} shows that actually $r=16$ suffice. In addition, $r=11$ is sufficient for piercing by 4 points.

We also show that the technique of Montejano and Sober\'{o}n can be used to obtain an alternative proof of the Hadwiger-Debrunner theorem, which may be of independent interest due to its simplicity.

\section{Proof of Theorem~\ref{Thm:Main-UBT}}
\label{sec:UBT}

As mentioned already, in the proof of Theorem~\ref{Thm:Main-UBT} we use Kalai's Upper Bound Theorem for convex sets~\cite{K84}, the Hadwiger-Debrunner theorem~\cite{HD57}, and the recent upper bound on $\HD_d(p,q)$ obtained in~\cite{KST16}. We state these results first.
\begin{theorem}[\cite{K84}]
\label{Thm:UBT}
Let $\F$ be a family of $p$ convex sets in $\Re^d$. Denote by $f_{q-1}$ the number of $q$-tuples of sets in $\F$
whose intersection is non-empty. If $f_{d+s}=0$ for some $s \geq 0$ then for any $q>0$,
\[
f_{q-1} \leq \sum_{i=0}^d {s \choose q-i} {p-s \choose i}.
\]
\end{theorem}

\begin{theorem}[\cite{HD57}]\label{Thm:HD}
For  $p \geq q \geq d+1$ such that $q > \frac{d-1}{d}p +1$,
$$HD_d(p,q) = p-q+1.$$
\end{theorem}

\begin{theorem}[\cite{KST16}]\label{Thm:Our}
Let $\eps>0$. There exists $p_0(\eps,d)$ such that for any $p \geq q \geq d+1$ with $p \geq p_0$ and $q \ge p^{\frac{d-1}{d}+\eps}$, we have
\[
\HD_d(p,q) \leq p-q+2.
\]
\end{theorem}

The intuition behind the proof is simple. Theorems~\ref{Thm:HD} and~\ref{Thm:Our} yield a strong bound on the piercing number for
a family that satisfies the $(p,q)$ property with a `large' $q$. In order to apply them, we need to `enlarge' $q$, and this is done
using Theorem~\ref{Thm:UBT}. Specifically,
if some family $\F'$ of $p$ convex sets contains `many' intersecting $q$-tuples, Theorem~\ref{Thm:UBT} allows
to deduce that it also contains an intersecting $(q+k)$-tuple, for an appropriate value of $k$. This implies that if a family $\F$ satisfies
the $(p,q)_r$ property, then it must satisfy the $(p,q+k)$ property, for $k=k(r)$. Applying this with a sufficiently
large $r$, we replace $q$ with a sufficiently large $q+k$, and then apply an improved bound on the piercing number that follows from
Theorem~\ref{Thm:HD} or Theorem~\ref{Thm:Our}.

\subsection{Proof of Theorem~\ref{Thm:Main-UBT}(1)}
\label{sec:sub:Thm(1)}

We need the following lemma:
\begin{lemma}\label{Prop:UBT(1)1}
Let $p \geq q \geq d+1$, and let $1 \leq f(p) \leq \frac{p}{d}-1$. If
\[
r \geq r_0 := \sum_{i=0}^d {{p-f(p)-d}\choose{q-i}} {{f(p)+d}\choose{i}} + 1,
\]
then $\HD_d(p,q)_r \leq f(p)$.
\end{lemma}

\begin{proof}
Let $\F$ be a family of compact convex sets in $\Re^d$ that satisfies the $(p,q)_r$ property for some $r \geq r_0$. Put $k=p-q-f(p)+1$. Note that $\frac{d-1}{d}p-q+2 \leq k \leq p-q$. By the choice of $r_0$, Theorem~\ref{Thm:UBT} implies that $\F$ satisfies the $(p,q+k)$ property. As $q+k \geq \frac{d-1}{d}p+2$, Theorem~\ref{Thm:HD} implies that the piercing number of $\F$ is at most $p-(q+k)+1=f(p)$, as asserted.
\end{proof}

\begin{proof}[Proof of Theorem~\ref{Thm:Main-UBT}(1)]
First, note that if $p-q+1 < \frac{p}{d}-1$, then $q > \frac{d-1}{d}p +2$, and thus Theorem~\ref{Thm:HD} implies $\HD_d(p,q)_r \leq p-q+1$ even for $r=1$. Hence, we may assume $\frac{p}{d}-1 \leq p-q+1$. Substituting $f(p)=\frac{p}{d}-1$ into Lemma~\ref{Prop:UBT(1)1}, we get $\HD_d(p,q)_r \leq f(p)=\frac{p}{d}-1$ for all
\begin{align*}
r &\geq \sum_{i=0}^d {{p-(\frac{p}{d}-1)-d}\choose{q-i}} {{(\frac{p}{d}-1)+d}\choose{i}} + 1 \\
&= {{\frac{d-1}{d}p+1-d}\choose{q}} + {{\frac{d-1}{d}p+1-d}\choose{q-1}} \cdot(\frac{p}{d}+d-1) + \ldots + {{\frac{d-1}{d}p+1-d}\choose{q-d}} {{\frac{p}{d}+d-1}\choose{d}} \\
&= O_d \left({{\frac{d-1}{d}p}\choose{q-d}} {{\frac{p}{d}}\choose{d}} \right),
\end{align*}
as asserted.
\end{proof}

\begin{remark}
Note that Lemma~\ref{Prop:UBT(1)1} actually supplies a sequence of upper bounds on $r$, which correspond to any desired piercing number between 1 and $\frac{p}{d}-1$. Piercing numbers larger than $\frac{p}{d}-1$ are treated in Section~\ref{sec:Soberon}.
\end{remark}

\subsection{Proof of Theorem~\ref{Thm:Main-UBT}(2)}
\label{sec:sub:Thm(2)}

The proof of Theorem~\ref{Thm:Main-UBT}(2) is similar to the proof of Theorem~\ref{Thm:Main-UBT}(1), with Theorem~\ref{Thm:Our} replacing Theorem~\ref{Thm:HD}. \begin{proof}[Proof of Theorem~\ref{Thm:Main-UBT}(2)]
Let $\eps>0$, and let $p>p_0(\eps)$ where $p_0$ is chosen to satisfy the hypothesis of Theorem~\ref{Thm:Our}.

First, consider the case $p-q+1 < p-p^{\frac{d-1}{d}+\eps}+1$. Recall that by assumption, $\F$ satisfies the $(p,q)_r$ property with
\begin{equation}\label{Eq:Aux02}
r \geq \Theta_{d,\eps} \left(\frac{p^{\left(\frac{d-1}{d}+\eps \right)q+1}}{(q-d)!} \right),
\end{equation}
and thus, also with
\[
r=\sum_{i=0}^d {{q-d}\choose{q-i}} {{p-q+d}\choose{i}}  + 1
\]
(actually, this is assured by taking the implicit factor in $\Theta_{d,\eps}(\cdot)$ to be sufficiently large). By Theorem~\ref{Thm:UBT}, the latter implies that $\F$ satisfies the $(p,q+1)$ property. As in this case, $q > p^{\frac{d-1}{d}+\eps}$, Theorem~\ref{Thm:Our} implies that $\F$ can be pierced with at most $p-(q+1)+2=p-q+1$ points, as asserted. Hence, we may assume $p - p^{\frac{d-1}{d}+\eps}+1 \leq p-q+1$.

Let $k=p^{\frac{d-1}{d}+\eps} - q$. Since by assumption, $\F$ satisfies the $(p,q)_r$ property with $r$ that satisfies~\eqref{Eq:Aux02}, in particular $\F$ satisfies the $(p,q)_r$ property with
\begin{align*}
r = &\sum_{i=0}^d {{p^{\frac{d-1}{d}+\eps}-d-1}\choose{q-i}}{{p-p^{\frac{d-1}{d}+\eps}+d+1}\choose{i}} + 1 \\
=&\sum_{i=0}^d {{q+k-d-1}\choose{q-i}} {{p-q-k+d+1}\choose{i}}  + 1.
\end{align*}
By Theorem~\ref{Thm:UBT}, this implies that $\F$ satisfies the $(p,q+k)$ property. As $q+k = p^{\frac{d-1}{d}+\eps}$, Theorem~\ref{Thm:Our} implies that $\F$ can be pierced by at most $p-p^{\frac{d-1}{d}+\eps}+2$ points. This completes the proof.
\end{proof}

\begin{remark}\label{Rem:r}
As in Section~\ref{sec:sub:Thm(1)}, a similar argument (using Theorem~\ref{Thm:Our} instead of Theorem~\ref{Thm:HD}) shows that for any $p>p_0(\eps)$ and any $1 \leq f(p) \leq p-p^{\frac{d-1}{d}+\eps}+2$, we have $\HD(p,q)_r \leq f(p)$ for all
\[
r> \sum_{i=0}^d {{p-f(p)+1-d}\choose{q-i}}{{f(p)-1+d}\choose{i}}.
\]
\end{remark}

The upper bound on $r$ asserted in Theorem~\ref{Thm:Main-UBT}(2) is not far from being optimal, as demonstrated by the following example (presented in~\cite{MS11}).
\begin{example*}
Let $\F$ be a family composed of $p-p^{\frac{d-1}{d}+\eps}+3$ pairwise disjoint sets and $p^{\frac{d-1}{d}+\eps}-3$ copies of a convex set that contains all of them. An easy computation shows that $\F$ satisfies the $(p,q)_r$ property for
\[
r={{p^{\frac{d-1}{d}+\eps}-3}\choose{q}} + (p-p^{\frac{d-1}{d}+\eps}+3) \cdot {{p^{\frac{d-1}{d}+\eps}-3}\choose{q-1}} = \Theta \left(\frac{p^{(\frac{d-1}{d}+\eps)(q-1)+1}}{(q-1)!}\right),
\]
while it clearly cannot be pierced by $p-p^{\frac{d-1}{d}+\eps}+2$ points.
\end{example*}
A similar example, with $p-f(p)+2$ instead of $p^{\frac{d-1}{d}+\eps}$, shows that the upper bound on $r$ asserted in Remark~\ref{Rem:r} is also near tight.

\medskip

Finally, we note that in dimension $1$, the \emph{exact relation} between the $(p,q)_r$ property and the piercing number can be obtained easily using the Upper Bound Theorem.
\begin{proposition}\label{Prop:Dim1}
For $p \geq q \geq 2$, let $\F$ be a family of segments on the real line that satisfies the $(p,q)_r$ property. If
\begin{equation}\label{Eq:Dim1}
r \geq {{p-k-2}\choose{q}} + (k+2){{p-k-2}\choose{q-1}}+1,
\end{equation}
then $\F$ can be pierced by $k+1$ points. Conversely, there exists a family $\F_0$ that satisfies the $(p,q)_{r}$ property with $r ={{p-k-2}\choose{q}} + (k+2){{p-k-2}\choose{q-1}}$ and cannot be pierced by $k+1$ points.
\end{proposition}

\begin{proof}
By Theorem~\ref{Thm:UBT}, if $\F$ satisfies the $(p,q)_r$ property with $r$ that satisfies~\eqref{Eq:Dim1}, then $\F$ satisfies the $(p,p-k)$ property. By Theorem~\ref{Thm:HD} this implies that $\F$ can be pierced by $k+1$ points.

For the other direction, let $\F_0$ be a family that consists of $k+2$ distinct single-point sets, and $p-k-2$ copies of a segment that contains all the points. A straightforward computation shows that $\F_0$ satisfies the $(p,q)_r$ property with $r ={{p-k-2}\choose{q}} + (k+2){{p-k-2}\choose{q-1}}$, but cannot be pierced by $k+1$ points.
\end{proof}
Proposition~\ref{Prop:Dim1} will be useful for us in the next section.

\section{Proof of Theorem~\ref{Thm:Main-Soberon}}
\label{sec:Soberon}

In the proof of Theorem~\ref{Thm:Main-Soberon} we use a bootstrapping based on the technique presented by Montejano and Sober\'{o}n~\cite{MS11}. First we state a lemma of~\cite{MS11} on which we base our argument.

\subsection{The technique of~\cite{MS11} and an alternative proof of the Hadwiger-Debrunner theorem}
\label{sec:sub:MS}

\begin{lemma}\label{Lem:MS}
For any family $\F$ of convex sets in $\Re^d$, there exist $A_1,A_2,\ldots,A_d \in \F$ and a line $\ell$ such that if $C \in \F$ intersects $\cap_{i \neq j} A_i$ for all $1 \leq j \leq d$ then $C \cap \ell \neq \emptyset$.
\end{lemma}

Since our argument is partially based on the proof of Lemma~\ref{Lem:MS} presented in~\cite{MS11}, we recall the proof below. In the general case of families in $\Re^d$, the proof of~\cite{MS11} uses topological techniques. As we do not use these parts of the proof of~\cite{MS11}, we present the proof in the case of $\Re^2$ where the topological tools are not needed, and refer the reader to~\cite[Theorem~2.62]{HellyToday} for sketch of the proof in the general case. For sake of clarity, we formulate explicitly the $d=2$ case of Lemma~\ref{Lem:MS} whose proof we present.
\begin{lemma*}[Lemma~\ref{Lem:MS} for $d=2$]
For any family $\F$ of convex sets in $\Re^2$, there exist $A,B \in \F$ and a line $\ell$ such that if $C \in \F$ satisfies $A \cap C \neq \emptyset$ and $B \cap C \neq \emptyset$ then $C \cap \ell \neq \emptyset$.
\end{lemma*}

\begin{proof}
If some $A,B \in \F$ satisfy $A \cap B = \emptyset$ then the assertion clearly holds with $A,B$ and any line $\ell$ that separates between $A$ and $B$. Thus, we assume that $A \cap B \neq \emptyset$ for all $A,B \in \F$.

For any pair $A,B \in \F$ such that $A \cap B \neq \emptyset$, let $\mathrm{lexmax}(A,B)$ denote the lexicographic maximum of $A \cap B$. Let $x_0 = \mathrm{lexmin} \{\mathrm{lexmax}(A,B): A,B \in \F, A \cap B \neq \emptyset\}$ (i.e., the lexicographic minimum amongst $\mathrm{lexmax}(A,B)$), and let $A,B \in \F$ be such that $x_0= \mathrm{lexmax} (A,B)$. Denote $H = \{x \in \Re^2: x \geq_{\mathrm{lex}} x_0\}$, and let $A'=A \cap H$, $B' = B \cap H$. As $A',B'$ are convex sets and $A' \cap B' = \{x_0\}$, there exists a line $\ell$ with $x_0 \in \ell$ that separates between $A' \setminus \{x_0\}$ and $B' \setminus \{x_0\}$. We claim that the assertion holds with $A,B,\ell$. To see this, we consider two cases:
\begin{enumerate}
\item $C \cap (A \cap B) \neq \emptyset$. We claim that $x_0 \in C$, and thus $C \cap \ell \neq \emptyset$. Assume to the contrary $x_0 \not \in C$. Note that for any family of convex sets $C_1,C_2,\ldots,C_m \subset \Re^2$ such that $\cap_{i=1}^m C_i \neq \emptyset$, there exist $1 \leq k < l \leq m$ such that $\mathrm{lexmax}(\cap_{i=1}^m C_i) = \mathrm{lexmax}(C_k \cap C_l)$. (This is a straightforward application of Helly's theorem; see~\cite{Matousek}, Lemma~8.1.2). In the case $\{C_1,\ldots,C_m\}=\{A,B,C\}$, by assumption $x_1:= \mathrm{lexmax}(A \cap B \cap C) \neq \mathrm{lexmax}(A \cap B)$, and thus w.l.o.g. $x_1 = \mathrm{lexmax}(A \cap C)$. It follows that $x_1 \in A \cap B$, and thus, $x_1 <_{\mathrm{lex}} x_0 = \mathrm{lexmax}(A \cap B)$. A contradiction to the definition of $x_0$. Hence, $x_0 \in C$, as asserted.

\item $C \cap (A \cap B) = \emptyset$. As $\mathrm{lexmax}(A \cap C) >_{\mathrm{lex}} x_0$, we have $(A' \setminus \{x_0\}) \cap C \neq \emptyset$. Similarly, we have $(B' \setminus \{x_0\}) \cap C \neq \emptyset$. As $\ell$ separates between $A' \setminus \{x_0\}$ and $B' \setminus \{x_0\}$, this implies $C \cap \ell \neq \emptyset$, as asserted.
\end{enumerate}
\end{proof}

\begin{remark}\label{Rem:Auxi}
When the proof of Case~(1) of Lemma~\ref{Lem:MS} is applied for a general $d$ (as was done in~\cite{MS11} and as we do below), we define $x_0 = \mathrm{lexmin} \{\mathrm{lexmax}(\cap_{i=1}^d A_i): A_1,\ldots,A_d \in \F, \cap_{i=1}^d A_i \neq \emptyset\}$ (where the system of coordinates is chosen such that all lexicographic maxima/minima are defined uniquely). We also replace $A \cap B$ by $\cap_{i=1}^d A_i$, and replace `each of $A$ and $B$' by `each of $\cap_{i \neq j} A_i$, $1 \leq j \leq d$'.
\end{remark}

The argument used in Case~(1) above can be used to obtain a simple proof of the Hadwiger-Debrunner theorem (Theorem~\ref{Thm:HD} above), as follows:
\begin{proof}[Alternative proof of Theorem~\ref{Thm:HD}]
Let $\F$ be a family of at least $p$ compact convex sets in $\Re^d$ that satisfies the $(p,q)$ property, and let $x_0,A_1,A_2,\ldots,A_d$ be chosen as in the proof of Lemma~\ref{Lem:MS} and in Remark~\ref{Rem:Auxi}.

Consider the family $\G=\{C \in \F: x_0 \not \in C\}$. We consider two cases:
\begin{itemize}
\item $|\G| \geq p-d$. We claim that in this case, $\G$ satisfies the $(p-d,q-d+1)$ property. Indeed, let $C_1,C_2,\ldots,C_{p-d} \in \G$, and consider the family $\{C_1,\ldots,C_{p-d},A_1,A_2,\ldots,A_d\}$. By assumption, it contains an intersecting $q$-tuple. This $q$-tuple cannot contain all of $A_1,A_2,\ldots,A_d$, as by the argument of Case~(1) above, each of $C_1,\ldots,C_{p-d}$ is disjoint with $\cap_{i=1}^d A_i$. Thus, $\{C_1,C_2,\ldots,C_{p-d}\}$ contains an intersecting $(q-d+1)$-tuple.

\item $|\G| = p-d-t$ for $t>0$. By the same reasoning as in the previous case, $\G$ contains an intersecting $(q-d-t+1)$-tuple that can be pierced by a single point, and thus, it can be trivially pierced by $(p-d-t)-(q-d-t+1)+1=p-q$ points.
\end{itemize}
As $\F \setminus \G$ is pierced by $x_0$, combining the two cases we get $\HD(p,q) \leq \max(\HD(p-2,q-1)+1,p-q+1)$. Since $\HD(d+1,d+1)=1$ by Helly's theorem, it follows by induction that if $q > \frac{d-1}{d}p +1$ then $\HD(p,q) \leq p-q+1$.
\end{proof}

\subsection{The bootstrapping technique}
\label{sec:sub:BS}

In~\cite{MS11}, the authors show that if $\F$ satisfies the $(p,q)_r$ property with $r> {{p}\choose{q}} - {{p+1-d}\choose{q+1-d}}$ then (in the notations of Lemma~\ref{Lem:MS}) the family $\G' = \{C \cap \ell: C \in \F, x_0 \not \in C\}$ satisfies the $(p-d,q-d+1)$ property, and thus, by Theorem~\ref{Thm:HD} in dimension 1, $\F$ can be pierced by $p-q+1$ points. In our bootstrapping argument, we show instead that the family $\G'$ satisfies the $(p-q+1,2)_{r'}$ property for a sufficiently large $r'$, and then an improved piercing number for $\F$ can be derived from Proposition~\ref{Prop:Dim1}. We will use the following.
\begin{definition}
Let $\F$ be a family of compact convex sets in $\Re^d$, $|\F| \geq p$, and let $\ell$ be a line. $\F$ is said to satisfy the $(p,q)_r$ property \emph{through $\ell$} if any $p$-tuple of sets in $\F$ contains at least $r$ $q$-tuples that intersect on $\ell$.
\end{definition}

\begin{lemma}\label{Lem:L}
If a family $\F$ satisfies the $(p,2)_{r_0}$ property through $\ell$ where $r_0={{p-k-2}\choose{2}} + (k+2){{p-k-2}\choose{1}}+1$, then $\F$ can be pierced by $k+1$ points.
\end{lemma}

\begin{proof}
Let $\h = \{C \in \F: C \cap \ell = \emptyset\}$, and denote $h=|\h|$. The family $\F' = \{C \cap \ell: C \in \F \setminus \h\}$ clearly satisfies the $(p-h,2)_{r_0}$ property. As
\begin{align*}
{{p-k-2}\choose{2}} + (k+2){{p-k-2}\choose{1}}+1 \geq &{{(p-h)-(k-h)-2}\choose{2}} \\
&+ ((k-h)+2){{(p-h)-(k-h)-2}\choose{1}}+1,
\end{align*}
it follows that $\F'$ is a family of segments on $\ell$ that satisfies the $(p-h,2)_{r'}$ property with $r' = {{(p-h)-(k-h)-2}\choose{2}} + ((k-h)+2){{(p-h)-(k-h)-2}\choose{1}}+1$. Thus, by Proposition~\ref{Prop:Dim1}, $\F'$ can be pierced by $k-h+1$ points, and thus, $\F$ can be pierced by $k+1$ points, as asserted.
\end{proof}

Now we are ready to prove Theorem~\ref{Thm:Main-Soberon}.

\begin{proof}[Proof of Theorem~\ref{Thm:Main-Soberon}]
Let $\F$ be a family that satisfies the assumption of the theorem, and let $x_0,A_1,A_2,\ldots,A_d,\ell$ be chosen as in the proof of Lemma~\ref{Lem:MS}, i.e., $\{A_i\}_{i=1}^d$ is the $d$-tuple in which the $\mathrm{lexmin} (\mathrm{lexmax}(\cdot))$ is attained. Denote $\F' = \{C \in \F: x_0 \not \in C\}$. We want to show that $\F'$ satisfies the $(p-q+1,2)_{r_0}$ property through $\ell$, where $r_0={{(p-q+1)-k-2}\choose{2}} + (k+2){{(p-q+1)-k-2}\choose{1}}+1$. By Lemma~\ref{Lem:L}, this would imply that $\F'$ can be pierced by $k+1$ points, and thus, $\F$ can be pierced by $k+2$ points, as asserted.

By the choice of $m_0(k)$, it is sufficient to show that $\F'$ satisfies the $(p-q+1,2)_{r'}$ property through $\ell$, where $r'={{m_0(k)+1}\choose{2}}$. Furthermore, by Lemma~\ref{Lem:MS} it is sufficient to show that among any $p-q+1$ elements of $\F$ there exist at least ${{m_0(k)+1}\choose{2}}$ distinct pairs of elements that intersect $\cap_{i \neq j} A_i$ for all $1 \leq j \leq d$.

As $\F$ is non-$(p-q)$-degenerate, we have $|\F'| \geq p-q+1$. Let $\C = \{C_1,C_2,\ldots,C_{p-q+1}\} \subset \F'$, and let $\D=\{D_1,D_2,\ldots,D_{q-d-1}\} \subset \F $ such that $\D \cap (\C \cup \{A_1,A_2,\ldots,A_d\}) = \emptyset$. We have $|\C \cup \D \cup \{A_1,A_2,\ldots,A_d\}|=p$, and thus, the family $\C \cup \D \cup \{A_1,A_2,\ldots,A_d\}$ contains at least $r$ intersecting $q$-tuples.

\medskip \noindent
Note that $q$-tuples of elements of $\C \cup \D \cup \{A_1,A_2,\ldots,A_d\}$ can be divided into three groups:
\begin{enumerate}
\item $q$-tuples that contain less than $d-1$ of the sets $A_1,\ldots,A_d$.

\item $q$-tuples that contain exactly $d-1$ of the sets $A_1,\ldots,A_d$.

\item $q$-tuples that contain all the sets $A_1,\ldots,A_d$.
\end{enumerate}
We observe that none of the intersecting $q$-tuples belong to the third group, as by the proof of Lemma~\ref{Lem:MS} above (specifically, by Lemma~8.1.2 in~\cite{Matousek} that applies for a general $d$), all elements of $\C$ are disjoint with $\cap_{i=1}^d A_i$, and $\D$ contains only $q-d-1$ elements. This implies that the total number of intersecting $q$-tuples is at most ${{p}\choose{q}}-{{p-d}\choose{q-d}}$. Furthermore, since $r$ satisfies~\eqref{Eq:Ex-Soberon1}, the number of \emph{non-intersecting} $q$-tuples in groups 1 and 2 is at most
\begin{align}\label{Eq:Aux00}
\begin{split}
&{{p}\choose{q}}-{{p-d}\choose{q-d}} - \left({{p}\choose{q}} - {{p-d+1}\choose{q-d+1}} + 1 + {{q-d-2+m_0(k)}\choose{q-d}} + {{q-d-1+m_0(k)}\choose{q-d+1}}\right) \\
&= {{p-d}\choose{q-d+1}} - \left({{q-d-2+m_0(k)}\choose{q-d}} + {{q-d-1+m_0(k)}\choose{q-d+1}}+1\right).
\end{split}
\end{align}
For each $\{S_1,S_2,\ldots,S_{q-d+1}\} \subset \C \cup \D$, we define a $d$-tuple
\begin{align*}
P_{\{S_1,\ldots,S_{q-d+1}\}} =& \{\{S_1,\ldots,S_{q-d+1},A_2,A_3,\ldots,A_{d}\},\{S_1,\ldots,S_{q-d+1},A_1,A_3,\ldots,A_{d}\} \\ &,\ldots,\{S_1,\ldots,S_{q-d+1},A_1,A_2,\ldots,A_{d-1}\}\}.
\end{align*}
Denote by $P'$ the set of all $\{S_1,S_2,\ldots,S_{q-d+1}\} \subset \C \cup \D$ for which all $d$ elements of $P_{\{S_1,\ldots,S_{d-q+1}\}}$ are intersecting. We claim that
\begin{equation}\label{Eq:P}
|P'| \geq {{q-d-2+m_0(k)}\choose{q-d}} + {{q-d-1+m_0(k)}\choose{q-d+1}}+1.
\end{equation}
Indeed, note that the $q$-tuples in group 2 are naturally divided into $d$ classes according to the set $A_i$ they miss. Each class consists of ${{p-d}\choose{q-d+1}}$ $q$-tuples. It is clear that for a given number of intersecting $q$-tuples, $|P'|$ is minimized when all non-intersecting $q$-tuples of group 2 belong to the same class. In that case, $|P'|$ equals to the number of remaining elements in that class, and thus by Equation~\eqref{Eq:Aux00},
\[
|P'| \geq {{p-d}\choose{q-d+1}} - \left({{p-d}\choose{q-d+1}} - \left({{q-d-2+m_0(k)}\choose{q-d}} + {{q-d-1+m_0(k)}\choose{q-d+1}}+1\right) \right),
\]
meaning that~\eqref{Eq:P} holds.

By the definition of $P'$, each element in $P'$ contains $q-d+1$ sets that intersect $\cap_{i \neq j} A_i$, for all $1 \leq j \leq d$. As $|\D|=q-d-1$, at least two of these sets belong to $\C$. Hence, each element of $P'$ contains at least one pair of elements in $\C$ that intersect $\cap_{i \neq j} A_i$, for all $1 \leq j \leq d$. Recall that we want to prove that there are at least ${{m_0(k)+1}\choose{2}}$ such pairs.

It is easy to see that for a given number of elements in $P'$, the number of distinct pairs $(C,C') \in \C$ contained in elements of $P'$ is minimized when these elements are `packed together'. In particular, the maximal possible number of elements in $P'$ such that the number of distinct pairs is \emph{smaller than}  ${{m_0(k)+1}\choose{2}}$ is attained when we take some $C_1,C_2,\ldots,C_{m_0(k)+1} \in \C$, and define
\begin{align}\label{Eq:P'}
\begin{split}
P''= \{S \subset \C \cup \D: \left(|S|=q-d+1 \right) \wedge &\left(S \cap \C \subset \{C_1,C_2,\ldots,C_{m_0(k)+1}\} \right) \wedge \\
&\wedge \left(\{C_{m_0(k)},C_{m_0(k)+1}\} \not \subset S \right) \}.
\end{split}
\end{align}
In this case, we have $|P''|={{q-d-2+m_0(k)}\choose{q-d}} + {{q-d-1+m_0(k)}\choose{q-d+1}}$. Indeed, since $|\D|=q-d-1$, then among the $(q-d+1)$-tuples $S$ for which $S \cap \C \subset \{C_1,C_2,\ldots,C_{m_0(k)+1}\}$, there are ${{q-d-1+m_0(k)}\choose{q-d+1}}$ that do not include $C_{m_0(k)+1}$, and ${{q-d-2+m_0(k)}\choose{q-d}}$ that include $C_{m_0(k)+1}$ and miss $C_{m_0(k)}$. Therefore, Equation~\eqref{Eq:P} implies that $\C$ must contain at least ${{m_0(k)+1}\choose{2}}$ distinct pairs that intersect $\cap_{i \neq j} A_i$ for all $1 \leq j \leq d$, and thus, by Lemma~\ref{Lem:MS}, $\F'$ satisfies the $(p-q+1,2)_{{{m_0(k)+1}\choose{2}}}$ property through $\ell$. This completes the proof.
\end{proof}

\section*{Acknowledgements}

The authors are grateful to Tsvi and Ronit Lubin for their programming help during the research.

\end{document}